\documentclass{amsart}
\usepackage{amsrefs}
\usepackage{amsthm}
\usepackage{graphicx,color}
\usepackage[english]{babel}  
\usepackage[latin1]{inputenc}
\usepackage[mathscr]{eucal}
\usepackage{amssymb}
\usepackage[dvipsnames]{xcolor} 
\begin{document}

\title[Nonlocal Schr\"odinger equations]
{Nonlocal Schr\"odinger equations for Integro-Differential Operators with Measurable Kernels and Asymptotic Potentials.}

\author[Ronaldo C. Duarte]{Ronaldo C. Duarte}

\address{Ronaldo C. Duarte \newline Departamento de Matem\'atica,
         Universidade Federal de Cam\-pi\-na Grande,
         58429-010, Campina Grande - PB, Brazil}

\email{ronaldocesarduarte@gmail.com}

\address{Marco A. S. Souto \newline Departamento de Matem\'atica,
	Universidade Federal de Cam\-pi\-na Grande,
	58429-010, Campina Grande - PB, Brazil}

\email{marco.souto.pb@gmail.com}

\author[Marco A. S. Souto]{Marco A. S. Souto}



\begin{abstract}
In this paper, we investigate the existence of nonnegative solutions for the problem
$$
-\mathcal{L}_{K}u+V(x)u=f(u)
$$
in $\mathbb R^n$, where $-\mathcal{L}_{K}$ is a integro-differential operator with measurable kernel $K$ and $V$ is a continuous potential. Under apropriate hypothesis, we prove, using variational methods, that the above equation has solution.
\end{abstract}

\maketitle

\newtheorem{theorem}{Theorem}[section]
\newtheorem{lemma}[theorem]{Lemma}
\newtheorem{proposition}[theorem]{Proposition}
\newtheorem{corollary}[theorem]{Corollary}
\newtheorem{remark}[theorem]{Remark}
\newtheorem{definition}[theorem]{Definition}
\renewcommand{\theequation}{\thesection.\arabic{equation}}

\section{Introduction}

In this article we consider the class of integro-differential Schr\"odinger equations
$$
\begin{array}{lcl}
- \mathcal{L}_{K} u +V(x)u = f(u),  &\mbox{ in }& \mathbb R^n,
\end{array}\leqno (P)
$$
where $- \mathcal{L}_{K}$ is a integro-differential operator given by
$$
- \mathcal{L}_{K}u(x)= 2 \int_{\mathbb{R}^{n}}(u(x)-u(y))K(x-y)dy
$$
and $K$  satisfy general properties. This study leads both to nonlocal and to nonlinear difficulties. For example, we can not benefit from the $s$-harmonic extension of \cite{caf} or commutator properties (see \cite{secchi}). 

The study of nonlocal operators is important because they intervene in a quantity of applications and models. For example, we mention their use in phase transition models (see \cite{al}, \cite{cab}), image reconstruction problems (see \cite{gil}), obstacle problem, optimization, finance, phase transitions. Integro-differential equations arise naturally in the study of stochastic processes with jumps, and more precisely of Lévy processes.

This paper was motivated by \cite{alv}, where the authors study the existence of positive solutions for the problem 
$$
\left\{\begin{array}{lcl}
- \Delta u +V(x)u = f(u),  &\mbox{ in }& \mathbb R^n, \\ u \in D^{1,2}(\mathbb{R}^{n})
\end{array}\right.
$$
where $V$ and $f$ are continuous functions with $V$ being a nonnegative function and $f$ having a subcritical or critical growth. Our purpose is to study an analogously problem, considering the operator $-\mathcal{L}_{k}$ instead of the Laplacian operator.

Several papers have studied the problem $(P)$ when $K (x) = \frac{C_{n,s}}{2}|x|^{n + 2s}$,
where 
$$
C_{n,s}=\left(\int_{\mathbb{R}^{n}}\frac{1-\cos(\xi_{1})}{|\xi|^{n+2s}}d \xi \right)^{-1},
$$ 
 that is, when the operator $-\mathcal{L}_{k}$ is the fractional Laplacian operator (see \cite{dpv}). Next, we will mention some of these papers. In \cite{vin}, the author has proved the existence of positive solutions from $(P)$ when $V$ is a constant small enough. Also, in \cite{raq}, the problem was studied when $f$ is asymptotically linear and $V$ is constant. In \cite{you}, the authors study the problem $(P)$ when $V \in C^{n}(\mathbb{R}^{n},\mathbb{R})$, $V$ is positive and 
 $$\lim\limits_{n \rightarrow \infty}V(|x|) \in (0, \infty].$$ In \cite{hui}, the authors has Studied $(P)$ when $ V $ and $ f $ are asymptotically periodic. When $V=1$, Felmer et al. has studied the existence,
regularity and qualitative properties of ground states solutions for problem $(P)$ (see \cite{fel}). In \cite{ten}, Teng and He have shown the existence of solution for $(P)$ when $$f(x,u)=P(x)|u|^{p-2}u+Q(x)|u|^{2^{\ast}_{s}-2}u,$$ where $2 < p < 2^{\ast}_{s}$ and the potential functions $P(x)$ and $Q(x)$ satisfy certain hypothesis. In \cite{jiafa}, the authors have shown the existence of solution for $(P)$ when $V \in C^{n}(\mathbb{R}^{n}, \mathbb{R})$ and there exists $r_{0} > 0$ such that, for any $M > 0$,
 $$\mbox{meas}(\left\{x \in B_{r_{0}}(y); V (x)\leq M\right\})\rightarrow 0\mbox{ as }|y| \rightarrow  \infty.$$ In \cite{secchi} the problem $(P)$ was studied when $V \in C^1(\mathbb{R}^{n},\mathbb{R})$, $\liminf_{|x|\rightarrow \infty}V(x)\geq V_{\infty}$ where $V_{\infty}$ is constant, and $f \in C^{1}(\mathbb{R}^{n}, \mathbb{R})$. By method of the Nehari manifold , Sechi has shown that the problem $(P)$ has a solution if $V\leq V_{\infty}$, but $V$ is not identically equal to $V_{\infty}$, where $V_{\infty}$ is a constant. Also in \cite{secchi}, Secchi have obtained the existence of ground state solutions of $(P)$ for general $s \in (0,1)$ when $V(x)\rightarrow \infty$ as $|x| \rightarrow \infty$. In \cite{zhang2}, the authors obtain the existence of a sequence of radial and non radial solutions for the problem $(P)$ when $V$ and $f$ are radial functions.  Some other interesting studies by variational methods of the problem $(P)$ can be found in \cite{vincenzo}, \cite{bisci}, \cite{xiaojun}, \cite{chang}, \cite{chen}, \cite{pietro}, \cite{mou}, \cite{tian}, \cite{sof}, \cite{edmundo}, \cite{secchi2}, \cite{yane}, \cite{kaimin}, \cite{qingxuan} and \cite{yang2}. Many of them use strong tools that we can not use here in our problem, as the $s$-harmonic extension and commutator properties.

Here, we will admit that the potential $V$ is continuous and satisfies,
\begin{itemize}
\item$(V_{1}-)$ $\inf_{x \in \mathbb{R}^{n}}V(x)>0;$
\item$(V_{2}-)$ $V(x)\leq V_{\infty}$ for same constant $V_{\infty}>0$ and for all $x \in B_{1}(0)$.
\end{itemize}
Note that, $(V_{1})$ implies that
\begin{itemize}
\item$(V_{3}-)$ There are $R>0$ and $\Lambda>0$ such that
$$
V(x)\geq \Lambda
$$
for all $|x|\geq R$.
\end{itemize}
Also, we will assume that
$f\in C(\mathbb R,\mathbb R)$ is a function satisfying:
\begin{itemize}
\item $(f_{1}-)$ $|f(s)|\leq c_{0}|s|^{p-1}$, for some constant $C>0$ and $p \in (2,2^{\ast}_{s})$;
\item$(f_{2}-)$ There is $\theta>2$ such that
$$
\theta F(s)\leq sf(s)
$$
for all $s>0$;
\item$(f_{3}-)$ $f(t)>0$ for all $t>0$ and $f(t)=0$ for all $t<0$. 

\end{itemize}

The kernel $K:\mathbb{R}^{n}\rightarrow \left(0, \infty\right)$ is a measurable function such that 
\begin{itemize}
\item$(K_{1}-)$ $K(x)=K(-x)$ for all $x \in \mathbb{R}^{n}$;
\item $(K_{2}-)$ There is $\lambda>0$ and $s \in (0,1)$ such that $\lambda \leq K(x)|x|^{n+2s}$ almost everywere in $\mathbb{R}^{n}$;
\item$(K_{3}-)$ $\gamma K \in L^{1}(\mathbb{R}^{n})$, where $\gamma(x)=\min\left\{|x|^{2},1\right\}$.
\end{itemize}
Note that, when $K(x)=\frac{C_{n,s}}{2}|x|^{-(n+2s)}$ we have that $-\mathcal{L}_{k}$ is a fractional laplacian, $(-\Delta)^{s}$. 

Our paper is organized as follows. In section $2$, we will present some properties of the space in which we will study the problem $(P)$. In section $3$, we will define an auxiliary problem and we will show that the functional energy associated with the auxiliary problem satisfies the condition of Palais-Smale. By difficulty nonlocal of the operator $-\mathcal{L}_{K}$, we will can not use the same technique used in \cite{alv}. Therefore, we will present an another technique to show this result.  In section $4$, we will prove that a general estimative for weak solution of
$$
-\mathcal{L}_{K}u+b(x)u=g(x,u),
$$
where $b\geq 0$, $|g(x,t)|\leq h(x)|t|$ and $h\in L^{q}(\mathbb{R}^{n})$ with $q>\frac{n}{2s}$. We will show that there is $M=M(q,||h||_{L^{q}})$ such that the solution $u$ satisfies
$$
||u||_{\infty} \leq M||u||_{2^{\ast}_{s}}.
$$
In \cite{alv2}, using the $s$-harmonic extension of \cite{caf}, the authors has shown the same estimate when $ -\mathcal{L}_{K}$ is the fractional Laplacian operator. In our case, we can not use the $s$-harmonic extension, because we do not have an analogously extension for integro-differential operators. In section $5$, we show our main result in this paper, the Theorem \ref{thm52}.

\section{Preliminaries}

Let $s \in (0,1)$, we denote by $H^{s}(\mathbb{R}^{n})$ the  fractional sobolev space. It is defined as 
$$
H^{s}(\mathbb{R}^{n}):=\left\{u \in L^{2}(\mathbb{R}^{n});\int_{\mathbb{R}^{n}}\int_{\mathbb{R}^{n}}\frac{(u(x)-u(y))^{2}}{|x-y|^{n+2s}}dxdy<\infty \right\}.
$$ 
The space $H^{s}(\mathbb{R}^{n})$ is a Hilbert space with the norm
$$
||u||_{H^{s}}=\left(\int_{\mathbb{R}^{n}}|u|^{2}dx+\int_{\mathbb{R}^{n}}\int_{\mathbb{R}^{n}}\frac{(u(x)-u(y))^{2}}{|x-y|^{n+2s}}dxdy\right)^{\frac{1}{2}}
$$

We define $X$ as the linear space of Lebesgue measurable functions from $\mathbb{R}^{n}$ to $\mathbb{R}$ such that any function $u$ in $X$ belongs to $L^{2}(\mathbb{R}^{n})$ and the function
$$
(x,y)\longmapsto(u(x)-u(y))\sqrt{K(x-y)}
$$
is in $L^{2}(\mathbb{R}^{n}\times \mathbb{R}^{n})$.
The function 
$$
||u||_{X}:= \left(\int_{\mathbb{R}^{n}}u^{2}dx+\int_{\mathbb{R}^{n}}\int_{\mathbb{R}^{n}}(u(x)-u(y))^{2}K(x-y)dxdy\right)^{\frac{1}{2}}
$$
defines a norm in $X$ and $(X, ||\cdot||_{X})$ is a Hilbert space.
By $(K_{2})$, the space $X$ is continuously embedded in $H^{s}(\mathbb{R}^{n})$. Therefore, $X$ is continuously embedded in $L^{p}(\mathbb{R}^{n})$ for $p \in \left[2,2^{\ast}_{s}\right]$, where $2^{\ast}_{s}=\frac{2n}{n-2s}$. If $\Omega\subset \mathbb{R}^{n}$, we define
$$
X_{0}(\Omega)=\left\{u \in X; u=0 \mbox{ in }  \space \mathbb{R}^{n}\setminus \Omega.\right\}.
$$
The space $X_{0}(\Omega)$ is a Hilbert Space with the norm
$$
u \longmapsto ||u||_{X_{0}(\Omega)}:=\left(\int_{\Omega}u^{2}dx+\int_{Q}\left(u(x)-u(y)\right)^{2}K(x-y)dxdy\right)^{\frac{1}{2}},
$$
where $Q=(\mathbb{R}^{n}\times\mathbb{R}^{n}) \setminus (\Omega^{c} \times \Omega^{c})$ (see 
Lemma $7$ in \cite{raf}). It is continuously embedded in $H_{0}^{s}(\mathbb{R}^{n})$. For definition and properties of $H_{0}^{s}(\mathbb{R}^{n})$ we indicate \cite{dpv}.

In the problem $(P)$ we will consider the space $E$ defined as
\begin{equation}\label{def21}
E=\left\{u \in X; \int_{\mathbb{R}^{n}}V(x)u^{2}dx< \infty \right\}
\end{equation}
The space $E$ is a Hilbert space with the norm
$$
u \longmapsto 
||u||:=\left(\int_{\mathbb{R}^{n}}\int_{\mathbb{R}^{n}}\left(u(x)-u(y)\right)^{2}K(x-y)dxdy+\int_{\mathbb{R}^{n}}V(x)u^{2}dx\right)^{\frac{1}{2}}.
$$
If $u,v \in C_{0}^{\infty}(\mathbb{R}^{n})$ then
$$
(-\mathcal{L}_{k}u,v)_{L^{2}}=[u,v].
$$
where
$$
[u,v]=\int_{\mathbb{R}^{n}}\int_{\mathbb{R}^{n}}\left(u(x)-u(y)\right)\left(v(x)-v(y)\right)K(x-y)dxdy.
$$
Therefore, we say that $u\in E$ is a solution for the problem $(P)$ if
$$
[u,v]+\int_{\mathbb{R}^{n}}V(x)uvdx = \int_{\mathbb{R}^{n}}f(u)vdx
$$
for all $v \in E$, that is
$$
\int_{\mathbb{R}^{n}}\int_{\mathbb{R}^{n}}\left(u(x)-u(y)\right)\left(v(x)-v(y)\right)K(x-y)dxdy+\int_{\mathbb{R}^{n}}V(x)uvdx = \int_{\mathbb{R}^{n}}f(u)vdx.
$$

Let $A, B \subset \mathbb{R}^{n}$ and $u,v \in X$. We will denote
$$
[u,v]_{A\times B}=\int_{A}\int_{B}(u(x)-u(y))(v(x)-v(y))K(x-y)dxdy
$$
and we will denote $[u,v]_{\mathbb{R}^{n}\times \mathbb{R}^{n}}$ by $[u,v]$.

The Euler-Lagrange functional associated with $(P)$ is given by
$$
I(u)=\frac{1}{2}||u||^{2}-\int_{\mathbb{R}^{n}}F(u)dx,
$$
where
$$
F(t)=\int_{0}^{t}f(s)ds.
$$
From hypothesis about $f$, the functional is $C^{1}(E,\mathbb{R})$ and
$$
I'(u)v=[u,v]+\int_{\mathbb{R}^{n}}V(x)uvdx-\int_{\mathbb{R}^{n}}f(u)vdx.
$$

We will denote by $B$ the unitary ball of $\mathbb{R}^{n}$. Define $I_{0}:X_{0}(B)\longrightarrow \mathbb{R}$ by
$$
I_{0}(u)=:\int_{\mathbb{R}^{n}} \int_{\mathbb{R}^{n}}\left(u(x)-u(y)\right)^{2}K(x-y)dxdy+\int_{\mathbb{R}^{n}}V_{\infty}u^{2}dx-\int_{\mathbb{R}^{n}}F(u)dx,
$$
where $V_{\infty}$ is the constant of $(V_{2})$. The functional $I_{0}$ has the mountain pass geometry. We will denote by $d$ the mountain pass level associated with $I_{0}$, that is
$$
d= \inf_{\gamma \in \Gamma} \max_{t \in [0,1]}I_{0}(\gamma(t)),
$$
where
\begin{equation}\label{eq1}
\Gamma=\left\{\gamma \in C([0,1],X_{0}(\Omega)); \gamma(0)=0\mbox{ and  }\gamma(1)=e\right\},
\end{equation}
with $e$ fixed and verifying $I_{0}(e)<0$. Note that $d$ depends only on $V_{\infty}$, $\theta$ and $f$.

\section{An Auxiliary Problem}
According to \cite{alv}, we will modified the problem defining an auxiliary problem. But, as the operator $-\mathcal{L}_{K}$ is nonlocal, we can not use the same ideas of \cite{alv} to prove that the functional associated the auxiliary problem satisfies the Palais-Smale condition. It is necessary that we use an another technics.

For $k=\frac{2 \theta}{\theta-2}$ we consider
$$
\begin{array}{lll}
\tilde{f}(x,t)&=
& \left\{
\begin{array}{lll}
f(t) & if & kf(t) \leq V(x)t \\
\frac{V(x)}{k}t& if & kf(t)>V(x)t
\end{array}
\right.
\end{array}
$$ 
and
\begin{equation}\label{eq32}
\begin{array}{lll}
g(x,t)&=
& \left\{
\begin{array}{lll}
f(t) & if & |x|\leq R \\
\tilde{f}(x,t)& if & |x|>R.
\end{array}
\right.
\end{array}
\end{equation}
And we define the auxiliary problem

$$
\left\{
\begin{array}{lllll}
-\mathcal{L}_{K}u+V(x)u&=&g(x,u)& in &\mathbb{R}^{n} \\
u \in E&&&&
\end{array}
\right.
$$
We have that, for all $t \in \mathbb{R}$ and $x \in \mathbb{R}^{n}$ 

\begin{enumerate}
\item $\tilde{f}(x,t)\leq f(t)$;
\item $g(x,t) \leq \frac{V(x)}{k}t$ ,se $|x|\geq R$;
\item $G(x,t)=F(t)$ se $|x|\leq R$
\item $G(x,t)\leq \frac{V(x)}{2k}t^{2}$ se $|x|>R$;
\end{enumerate}
where
$$
G(x,t)=\int_{0}^{t}g(x,s)ds.
$$
The Euler-Lagrange functional associated with the auxiliary problem is given by
$$
J(u)=\frac{1}{2}||u||^{2} - \int_{\mathbb{R}^{n}}G(x,u)dx.
$$
The functional $J \in C^{1}(X,\mathbb{R})$ and 
$$
\begin{array}{ll}
J'(u)v&=\int_{\mathbb{R}^{n}}\int_{\mathbb{R}^{n}}\left(u(x)-u(y)\right)\left(v(x)-v(y)\right)K(x-y)dxdy \\&+\int_{\mathbb{R}^{n}}V(x)uvdx - \int_{\mathbb{R}^{n}}g(x,u)vdx.
\end{array}
$$
The functional $J$ has the mountain pass geometry. Then, there is a sequence $\left\{u_{n}\right\}_{n \in \mathbb{N}}$ such that
\begin{equation}\label{def34}
J'(u_{n})\rightarrow 0
\mbox{ and } 
J(u_{n})\rightarrow c,
\end{equation}
where $c>0$ is the mountain pass level associated with $J$, that is
\begin{equation}\label{def35}
c= \inf_{\gamma \in \Gamma} \max_{t \in [0,1]}J(\gamma(t))
\end{equation}
where
$$
\Gamma=\left\{\gamma \in C([0,1],E); \gamma(0)=0\mbox{ and }, \gamma(1)=e\right\}.
$$
and $e$ is the function fixed in $\ref{eq1}$. By definition
\begin{equation}\label{eq:7}
c \leq d
\end{equation}
uniformly in $R>0$.
\begin{lemma}\label{lem31}
The sequence $\left\{u_{n}\right\}_{n \in \mathbb{N}}$ is bounded.
\end{lemma}

\begin{proof}
By $(f_{2})$, $(3)$ and $(4)$ 
$$
\begin{array}{ll}
&J(u)-\frac{1}{\theta}J'(u)u\\
& = \left(\frac{\theta-2}{4\theta}\right)||u||^{2}+ \frac{1}{2k}||u||^{2} + \int_{\mathbb{R}^{n}}\frac{1}{\theta}g(x,u)u-G(x,u)dx \\
&\geq \left(\frac{\theta-2}{4\theta}\right)||u||^{2}+ \frac{1}{2k}||u||^{2} + \int_{|x|>R}\frac{1}{\theta}g(x,u)u - \frac{1}{2k}\int_{|x|>R}V(x)u^{2}dx \\
& \geq \left(\frac{1}{2k}\right)||u||^{2}.
\end{array}
$$
Thereby
\begin{equation}\label{eq:8}
|J(u)|+|J'(u)u|\geq \left(\frac{\theta-2}{4\theta}\right)||u||^{2}
\end{equation}
for all $u \in E$.
This last inequality ensures that the sequence is bounded.
\end{proof}
 
Let $r>R$ and $A=\left\{x \in \mathbb{R}^{n}; r<||x||<2r\right\}$. Consider $\eta:\mathbb{R}^{n}\rightarrow \mathbb{R}$ a function such that $\eta=1$ in $B_{2r}^{c}(0)$, $\eta=0$ in $B_{r}(0)$, $0\leq\eta\leq1$ and $|\nabla \eta|<\frac{2}{r}$.
Note that 
\begin{equation}\label{eq:1}
(B_{r} \times B_{r})^{c} = (B_{r}^{c}\times \mathbb{R}^{n}) \cup ( B_{r} \times B_{r}^{c}),
\end{equation}
where $B_{r}=B_{r}(0)$ and $B_{2r}=B_{2r}(0)$. We will decompose
\begin{equation}\label{eq:2}
 B_{r}^{c}\times \mathbb{R}^{n} = (A \times \mathbb{R}^{n}) \cup (B_{2r}^{c}\times B_{r}) \cup (B_{2r}^{c}\times A) \cup (B_{2r}^{c}\times B_{2r}^{c})
\end{equation}
and
\begin{equation}\label{eq:3}
B_{r} \times B_{r}^{c} = (B_{r} \times A) \cup (B_{r} \times B_{2r}^{c})
\end{equation} 

\begin{lemma}\label{lem32}
We have that
$$
\begin{array}{ll}
&\int_{B_{r}}\int_{B_{2r}^{c}}(u_{n}(x)-u_{n}(y))(\eta(x)u_{n}(x)-\eta(y)u_{n}(y))K(x-y)dxdy\\
&+\int_{B_{2r}^{c}}\int_{B_{r}}(u_{n}(x)-u_{n}(y))(\eta(x)u_{n}(x)-\eta(y)u_{n}(y))K(x-y)dxdy \\
& \geq-\int_{B_{r}}\int_{B_{2r}^{c}}u_{n}(y)^{2}K(x-y)dxdy
\end{array}
$$
\end{lemma}
\begin{proof}
$$
\begin{array}{ll}
&\int_{B_{r}}\int_{B_{2r}^{c}}(u_{n}(x)-u_{n}(y))(\eta(x)u_{n}(x)-\eta(y)u_{n}(y))K(x-y)dxdy \\
& +\int_{B_{2r}^{c}}\int_{B_{r}}(u_{n}(x)-u_{n}(y))(\eta(x)u_{n}(x)-\eta(y)u_{n}(y))K(x-y)dxdy \\
& =\int_{B_{r}}\int_{B_{2r}^{c}}u_{n}(x)(u_{n}(x)-u_{n}(y))K(x-y)dxdy\\
&-  \int_{B_{2r}^{c}}\int_{B_{r}}u_{n}(y)(u_{n}(x)-u_{n}(y))K(x-y)dxdy \\
&= \int_{B_{r}}\int_{B_{2r}^{c}}u_{n}(x)(u_{n}(x)-u_{n}(y))K(x-y)dxdy\\
&-  \int_{B_{r}}\int_{B_{2r}^{c}}u_{n}(y)(u_{n}(x)-u_{n}(y))K(x-y)dydx \\
& = \int_{B_{r}}\int_{B_{2r}^{c}}(u_{n}(x)-u_{n}(y))^{2}K(x-y)dxdy \\
&+ \int_{B_{r}}\int_{B_{2r}^{c}}(u_{n}(y)+u_{n}(x))(u_{n}(x)-u_{n}(y))K(x-y)dxdy\\
& = \int_{B_{r}}\int_{B_{2r}^{c}}(u_{n}(x)-u_{n}(y))^{2}K(x-y)dxdy \\
&+ \int_{B_{r}}\int_{B_{2r}^{c}}u_{n}(x)^{2}-u_{n}(y)^{2}K(x-y)dxdy \\
& \geq -\int_{B_{r}}\int_{B_{2r}^{c}}u_{n}(y)^{2}K(x-y)dxdy
\end{array}
$$
\end{proof}
\begin{lemma}\label{lem33}
Let $\epsilon>0$. There is $r_{0}>0$ such that if $r>r_{0}$ then
$$
\int_{B_{r}}\int_{B_{2r}^{c}}u_{n}(y)^{2}K(x-y)dxdy<\epsilon,
$$
for all $n \in \mathbb{N}$.
\end{lemma}

\begin{proof}
For each $y\in B_{r}(0)$
$$
B_{r}(y) \subset B_{2r}(0).
$$
Then
\begin{equation}\label{def311}
\begin{array}{ll}
&\int_{B_{2r}(0)^{c}}K(x-y)dx \\
&\leq \int_{B_{r}(y)^{c}}K(x-y)dx\\
& = \int_{B_{r}(0)^{c}}K(z)dz. \\
\end{array}
\end{equation}
By Lemma \ref{lem31}, there is $L>0$ such that $||u_{n}||_{L^{2}}^{2}<L$ for all $n \in \mathbb{N}$.
By $(K_{3})$, there is $r_{0}>0$ such that 
$$
\int_{B_{r}(0)^{c}}K(z)dz<\frac {\epsilon}{L},
$$
for all $r>r_{0}$. Then by \ref{def311}, for all $n \in \mathbb{N}$ and $r>r_{0}$
$$
\begin{array}{ll}
&\int_{B_{r}(0)}\int_{B_{2r}(0)^{c}}u_{n}(y)^{2}K(x-y)dxdy\\
& =\int_{B_{r}(0)}u_{n}(y)^{2}\int_{B_{2r}(0)^{c}}K(x-y)dxdy\\
& \leq \int_{B_{r}(0)}u_{n}(y)^{2}\int_{B_{r}(0)^{c}}K(z)dzdy \\
& = \int_{B_{r}(0)^{c}}K(z)dz\int_{B_{r}(0)}u_{n}(y)^{2}dy\\
& \leq \epsilon
\end{array}
$$
\end{proof}

\begin{lemma}\label{lem34}
There are constants $K_{1}>0$ and $K_{2}>0$ such that
$$
\begin{array}{ll}
&\int_{A}\int_{\mathbb{R}^{n}}|u_{n}(y)||(u_{n}(x)-u_{n}(y))||(\eta(x)-\eta(y))|K(x-y)dxdy \\&\leq \frac{2K_{1}}{r}||u_{n}||_{L^{2}(A)}[u_{n},u_{n}]^{\frac{1}{2}} + K_{2}||u_{n}||_{L^{2}(A)}[u_{n},u_{n}]^{\frac{1}{2}}.
\end{array}
$$
\end{lemma}
\begin{proof}

Note that
$$
\begin{array}{ll}
\int_{\mathbb{R}^{n}}|\eta(x)-\eta(y)|^{2}K(x-y)dx& = \int_{\mathbb{R}^{n}}|\eta(z+y)-\eta(y)|^{2}K(z)dz \\
& =\int_{B_{1}(0)}|\eta(z+y)-\eta(y)|^{2}K(z)dz\\
&+\int_{B_{1}^{c}(0)}|\eta(z+y)-\eta(y)|^{2}K(z)dz \\
& \leq \frac{4}{r^{2}}\int_{B_{1}(0)}|z|^{2}K(z)dz +4\int_{B_{1}(0)^{c}}K(z)dz \\
& \leq \frac{4}{r^{2}}P_{1}+4P_{2},
\end{array}
$$
where
$$
P_{1}=\int_{B_{1}}|z|^{2}K(z)dz 
$$ and
$$
P_{2}= \int_{B_{1}^{c}}K(z)dz.
$$
Let $K_{1}=2\sqrt{P_{1}}$ and $K_{2}=2\sqrt{P_{2}}$. Then, by Holder inequality
$$
\begin{array}{ll}
&\int_{A}\int_{\mathbb{R}^{n}}|u_{n}(y)||(u_{n}(x)-u_{n}(y))||(\eta(x)-\eta(y))|K(x-y)dxdy \\
& \leq (\frac{2\sqrt{P_{1}}}{r}+2\sqrt{P_{2}})\int_{A}|u_{n}(y)|\left(\int_{\mathbb{R}^{n}}|(u_{n}(x)-u_{n}(y))|^{2}K(x-y)dx\right)^{\frac{1}{2}}dy\\
& \leq (\frac{K_{1}}{r}+K_{2})||u_{n}||_{L^{2}(A)}[u_{n},u_{n}]^{\frac{1}{2}}.
\end{array}
$$
\end{proof}

\begin{lemma}\label{lem35}
For the same constants $K_{1}>0$ and $K_{2}>0$ of the Lemma \ref{lem34}
$$
\begin{array}{ll}
&\int_{B_{r}}\int_{A}|u_{n}(x)-u_{n}(y)||\eta(x)u_{n}(x)-\eta(y)u_{n}(y)|K(x-y)dxdy \\&\leq \frac{K_{1}}{r}||u_{n}||_{L^{2}(A)}[u_{n},u_{n}]^{\frac{1}{2}} + K_{2}||u_{n}||_{L^{2}(A)}[u_{n},u_{n}]^{\frac{1}{2}}.
\end{array}
$$
\end{lemma}

\begin{proof}
Indeed, by property $(K_{1})$ of $K$
$$
\begin{array}{ll}
&\int_{B_{r}}\int_{A}|u_{n}(x)-u_{n}(y)||\eta(x)u_{n}(x)-\eta(y)u_{n}(y)|K(x-y)dxdy\\
& = \int_{B_{r}}\int_{A}|u_{n}(x)||u_{n}(x)-u_{n}(y)||\eta(x)|K(x-y)dxdy \\
& = \int_{A}\int_{B_{r}}|u_{n}(x)||u_{n}(x)-u_{n}(y)||n(x)-n(y)|K(x-y)dydx\\
& =\int_{A}\int_{B_{r}}|u_{n}(y)||u_{n}(y)-u_{n}(x)||n(y)-n(x)|K(y-x)dxdy

\end{array}
$$
Using $(K_{1})$ and Lemma \ref{lem34}, we prove this lemma.
\end{proof}

\begin{lemma}\label{ref36}
We have that
$$
\begin{array}{ll}
&-\int_{B_{2r}^{c}}\int_{A}u_{n}(y)(u_{n}(x)-u_{n}(y))(\eta(x)-\eta(y))K(x-y)dxdy \\
&\leq \frac{K_{1}}{r}||u_{n}||_{L^{2}(A)}[u_{n},u_{n}]^{\frac{1}{2}} + K_{2}||u_{n}||_{L^{2}(A)}[u_{n},u_{n}]^{\frac{1}{2}}.
\end{array}
$$
\end{lemma}

\begin{proof}
$$
\begin{array}{ll}
&-\int_{B_{2r}^{c}}\int_{A}u_{n}(y)(u_{n}(x)-u_{n}(y))(\eta(x)-\eta(y))K(x-y)dxdy\\
&=\int_{B_{2r}^{c}}\int_{A}(u_{n}(x)-u_{n}(y))^{2}(\eta(x)-\eta(y))K(x-y)dxdy \\
&-\int_{B_{2r}^{c}}\int_{A}u_{n}(x)(u_{n}(x)-u_{n}(y))(\eta(x)-\eta(y))K(x-y)dxdy \\
&  = \int_{B_{2r}^{c}}\int_{A}(u_{n}(x)-u_{n}(y))^{2}(n(x)-1)K(x-y)dxdy \\
&-\int_{B_{2r}^{c}}\int_{A}u_{n}(x)(u_{n}(x)-u_{n}(y))(n(x)-n(y))K(x-y)dxdy \\
& \leq -\int_{B_{2r}^{c}}\int_{A}u_{n}(x)(u_{n}(x)-u_{n}(y))(n(x)-n(y))K(x-y)dxdy \\
& = -\int_{B_{2r}^{c}}\int_{A}u_{n}(x)(u_{n}(y)-u_{n}(x))(n(y)-n(x))K(x-y)dxdy \\
& \leq \int_{B_{2r}^{c}}\int_{A}|u_{n}(x)||u_{n}(y)-u_{n}(x)||n(y)-n(x)|K(x-y)dxdy \\
& =\int_{A}\int_{B_{2r}^{c}}|u_{n}(x)||u_{n}(y)-u_{n}(x)||n(y)-n(x)|K(y-x)dydx \\
& \leq \frac{K_{1}}{r}||u_{n}||_{L^{2}(A)}[u_{n},u_{n}]^{\frac{1}{2}} + K_{2}||u_{n}||_{L^{2}(A)}[u_{n},u_{n}]^{\frac{1}{2}}.
\end{array}
$$
In the last inequality, we have used the Lemma \ref{lem34} and $(K_{1})$.
\end{proof}

We will prove that the functional $J$ satisfies the Palais-Smale condition. The next proposition ensures the existence of solution in the level $c$ for the auxiliary problem (see \ref{def35}). We emphasize that by a nonlocal difficulty, we can not repeat the same arguments used in \cite{alv} to show that the functional energy associated at the auxiliary problem satisfies the Palais-Smale condition, therefore we use another technique to show this.
 
  \begin{proposition}\label{prop37}
  Suppose that $f$ and $V$ satisfy $(V_{1})$, $(f_{1})-(f_{3})$. Then the functional $J$ satisfies the Palais-Smale condition.
  \end{proposition}
\begin{proof} 
By Lemma \ref{lem31} the Palais-Smale sequence $\left\{u_{n}\right\}_{n \in \mathbb{N}}$ is bounded. We can suppose that $\left\{u_{n}\right\}_{n \in \mathbb{N}}$ converges weakly to $u$. 
By $K$ properties we have that $\eta u_{n} \in X$ and $||\eta u_{n}|| \leq ||u_{n}||$ (see Lemma $5.1$ in \cite{dpv}). Then, the sequence $\left\{\eta u_{n}\right\}$ is bounded in $X$. Therefore $I^{'}(u_{n})(\eta u_{n})=o_{n}(1)$. That is
$$
\begin{array}{ll}
&[u_{n},\eta u_{n}]+\int_{\mathbb{R}^{n}}V(x)u_{n}^{2}\eta dx = \int_{\mathbb{R}^{n}}g(x,u_{n})\eta u_{n}dx+ o_{n}(1).
\end{array}
$$
But, note that $[u_{n},\eta u_{n}] = [u_{n},\eta u_{n}]_{C(B_{r}\times B_{r})}$, because $\eta=0$ in $B_{r}$. By \ref{eq:1}, \ref{eq:2} and \ref{eq:3} we have:
$$
\begin{array}{ll}
&[u_{n},\eta u_{n}]_{A \times \mathbb{R}^{n}}+[u_{n},\eta u_{n}]_{B_{2r}^{c}\times A}+[u_{n},\eta u_{n}]_{B_{2r}^{c}\times B_{2r}^{c}}\\
&+[u_{n},\eta u_{n}]_{B_{2r}^{c}\times B_{r}}+[u_{n},\eta u_{n}]_{B_{r}\times B_{2r}^{c}}+[u_{n},\eta u_{n}]_{B_{r}\times A}\\
&+\int_{\mathbb{R}^{n}}V(x)u_{n}^{2}\eta dx = \int_{\mathbb{R}^{n}}g(x,u_{n})\eta u_{n}dx+ o_{n}(1)
\end{array}
$$
By Lemma \ref{lem32},
$$
\begin{array}{ll}
&[u_{n},\eta u_{n}]_{A \times \mathbb{R}^{n}}+[u_{n},\eta u_{n}]_{B_{2r}^{c}\times A}+\int_{\mathbb{R}^{n}}V(x)u_{n}^{2}\eta dx \\
&\leq  \int_{\mathbb{R}^{n}}g(x,u_{n})\eta u_{n}dx+ \int_{B_{r}}\int_{B_{2r}^{c}}u_{n}(y)^{2}K(x-y)dxdy-[u_{n},\eta u_{n}]_{B_{r}\times A}+o_{n}(1)
\end{array}
$$
Above, we have used that $[u_{n},\eta u_{n}]_{B_{2r}^{c}\times B_{2r}^{c}} =[u_{n}, u_{n}]_{B_{2r}^{c}\times B_{2r}^{c}} \geq 0$. Because $\eta = 1$ in $B_{2r}^{c}$.
If $C$ and $D$ are subsets of $\mathbb{R}^{n}$ and $u \in E$, then
$$
\begin{array}{ll}
[u, \eta u]_{C \times D}&=\int_{C}\int_{D}(u(x)-u(y))(\eta u(x)-\eta u(y))K(x-y)dxdy \\
&=\int_{C}\int_{D}\eta(x)(u(x)-u(y))^{2}K(x-y)dxdy\\
& + \int_{C}\int_{D}u(y)(u(x)-u(y))(\eta(x)-\eta(y))K(x-y)dxdy.
\end{array}
$$
Thereby,
$$
\begin{array}{ll}
&\int_{A}\int_{\mathbb{R}^{n}}\eta(x)(u_{n}(x)-u_{n}(y))^{2}K(x-y)dxdy+\\&+\int_{B_{2r}^{c}}\int_{A}\eta(x)(u_{n}(x)-u_{n}(y))^{2}K(x-y)dxdy +\int_{\mathbb{R}^{n}}V(x)u_{n}^{2}\eta dx \\
& \leq  \int_{\mathbb{R}^{n}}g(x,u_{n})\eta u_{n}dx+ \int_{B_{r}}\int_{B_{2r}^{c}}u_{n}(y)^{2}K(x-y)dxdy-[u_{n},\eta u_{n}]_{B_{r}\times A} \\
& -\int_{A}\int_{\mathbb{R}^{n}}u_{n}(y)(u_{n}(x)-u_{n}(y))(\eta(x)-\eta(y))K(x-y)dxdy \\
& -\int_{B_{2r}^{c}}\int_{A}u(y)(u_{n}(x)-u_{n}(y))(\eta(x)-\eta(y))K(x-y)dxdy+o_{n}(1).
\end{array}
$$
From Lemmas \ref{lem34}, \ref{lem35} and \ref{ref36}, we obtain constants $K_{1},K_{2}>0$ such that
$$
\begin{array}{ll}
&\int_{\mathbb{R}^{n}}V(x)u_{n}^{2}\eta dx \\
&\leq \int_{A}\int_{\mathbb{R}^{n}}\eta(x)(u_{n}(x)-u_{n}(y))^{2}K(x-y)dxdy\\
&+\int_{B_{2r}^{c}}\int_{A}\eta(x)(u_{n}(x)-u_{n}(y))^{2}K(x-y)dxdy +\int_{\mathbb{R}^{n}}V(x)u_{n}^{2}\eta dx \\
& \leq \int_{\mathbb{R}^{n}}g(x,u_{n})\eta u_{n}dx+ \int_{B_{r}}\int_{B_{2r}^{c}}u_{n}(y)^{2}K(x-y)dxdy \\
&+\frac{K_{1}}{r}||u_{n}||_{L^{2}(A)}[u_{n},u_{n}]^{\frac{1}{2}} + K_{2}||u_{n}||_{L^{2}(A)}[u_{n},u_{n}]^{\frac{1}{2}}+o_{n}(1).
\end{array}
$$
By $(2)$, $(f_{3})$ and $r>R$ we have
$$
 \int_{\mathbb{R}^{n}}g(x,u_{n})\eta u_{n}dx \leq \frac{1}{k}\int_{\mathbb{R}^{n}}\eta V(x)u_{n}^{2}dx.
$$
Thereby,
$$
\begin{array}{ll}
&\left(1-\frac{1}{k}\right)\int_{\mathbb{R}^{n}}V(x)u_{n}^{2}\eta dx  \\
& \leq \int_{B_{r}}\int_{B_{2r}^{c}}u_{n}(y)^{2}K(x-y)dxdy \\
&+\frac{K_{1}}{r}||u_{n}||_{L^{2}(A)}[u_{n},u_{n}]^{\frac{1}{2}} + K_{2}||u_{n}||_{L^{2}(A)}[u_{n},u_{n}]^{\frac{1}{2}}+o_{n}(1)
\end{array}
$$
By Lemma \ref{lem31}, there is $C_{1}>0$ such that $||u_{n}||\leq C_{1}$. Then,
for some constant $C>0$
\begin{equation}\label{eq:4}
\begin{array}{ll}
&\left(1-\frac{1}{k}\right)\int_{|x|>2r}V(x)u_{n}^{2}dx  \\
& \leq \int_{B_{r}}\int_{B_{2r}^{c}}u_{n}(y)^{2}K(x,y)dxdy +C(\frac{1}{r}+1)||u_{n}||_{L^{2}(A)}+o_{n}(1)
\end{array}
\end{equation}
Let $\epsilon>0$. By Lemma \ref{lem33}, we can take $r$, large enough, such that
\begin{equation}\label{eq:5}
\begin{array}{ll}
\int_{|x|>2r}V(x)u_{n}^{2}dx \leq \frac{\epsilon (k-1)}{3k}+C(\frac{1}{r}+1)||u_{n}||_{L^{2}(A)} +o_{n}(1)
\end{array}
\end{equation}
for all $n \in \mathbb{N}$. Also, we can assume that
\begin{equation}\label{eq:6}
||u||_{L^{2}(A)}< \frac{\epsilon (k-1)}{6 C(\frac{1}{r}+1)k}
\end{equation}
By property $(2)$ of $g$
$$
g(x,u_{n})u_{n} \leq \frac{V(x)}{k}u_{n}^{2}.
$$
for all $x$, with $|x|>2r>R$.
Therefore, by \ref{eq:4}
$$
\int_{|x|>2r}g(x,u_{n})u_{n}dx \leq \frac{\epsilon}{3}+C\left(\frac{1}{r}+1\right)\frac{k}{k-1}||u_{n}||_{L^{2}(A)} +o_{n}(1). 
$$
By \ref{eq:5}, \ref{eq:6} and compact embedding of the fractional Sobolev spaces, we can take $n_{1}\in \mathbb{N}$ such that if $n>n_{1}$ then
$$
\int_{|x|>2r}g(x,u_{n})u_{n}dx \leq \frac{5\epsilon}{6}. 
$$
Note that, we can suppose that $r>0$ satisfies
$$
\int_{|x|>2r}g(x,u)udx \leq \frac{\epsilon}{12}. 
$$
By compact embedding of fractional sobolev spaces we have that, there is, $n_{0} \in \mathbb{N}$ with $n_{0}>n_{1}$ and if $n>n_{0}$ then
$$
\left|
\int_{|x|\leq 2r}g(x,u_{n})u_{n}dx - \int_{|x|\leq 2r}g(x,u)udx\right|< \frac{\epsilon}{12}.
$$
Thereby, for $n>n_{0}$ we have
$$
\left|
\int_{\mathbb{R}^{n}}g(x,u_{n})u_{n}dx - \int_{\mathbb{R}^{n}}g(x,u)udx\right|< \epsilon
$$
that is
$$
\lim\limits_{n \rightarrow \infty}\int_{\mathbb{R}^{n}}g(x,u_{n})u_{n}dx = \int_{\mathbb{R}^{n}}g(x,u)udx
$$
From $I'(u_{n})u_{n}=o_{n}(1)$, we conclude that $||u_{n}|| \rightarrow ||u||$ and therefore $\left\{u_{n}\right\}$ converges to $u$ in $X$.
\end{proof}
\begin{corollary}\label{cor38}
Suppose $(V_{1})$, $(f_{1})-(f_{3})$. Then, there is $u \in X$ such that
$J(u)=c$ and $J'(u)=0$. Moreover, $u\geq 0$ almost everywere in $\mathbb{R}^{n}$.
\end{corollary}
\begin{proof}
By \ref{def34} and Proposition \ref{prop37}, there is $u \in X$ such that $J(u)=c$ and $J'(u)=0$. Let $A=\left\{x \in \mathbb{R}^{n};|x|>R;\right\} \cap \left\{x \in \mathbb{R}^{n}; u(x)<0;\right\}$. If $x \in A$, then $g(x,u(x))=\frac{V(x)}{k}u(x)$ and in if $x \in A^{c}$, then $g(x,u)\geq 0$. We have
$$
0\geq [u,u^{-}]+\int_{A^{c}}V(x)uu^{-} = \left(\frac{1}{k}-1\right)\int_{A}V(x)uu^{-} + \int_{A^{c}}g(x,u)u^{-}dx \geq 0
$$
where $u^{-}(x)=\max\left\{-u(x),0\right\}$. Then
$$
[u,u^{-}]=0
$$
This implies that $u^{-}=0$ (see proof of Lemma 4.1 in \cite{rm}).
\end{proof}
As a consequence of inequalities \ref{eq:7} and \ref{eq:8} we have the following proposition.
\begin{proposition}\label{prop39}
If $V$ and $f$ satisfies $(V_{1}),(V_{2}), (f_{1})-(f_{3})$, then the solution $u$ of the auxiliary problem satisfies
$$
||u||^{2} \leq 2kd 
$$
uniformly in $R>0$.
\end{proposition}

\section{$L^{\infty}$ estimative for solution of auxiliary problem}
In this section, we will prove a Brezis-Kato estimative. We will prove that, admitting some hypothesis, there is $M>0$ such that the solution of the problem 
$$
\left\{\begin{array}{rlll}
	-\mathcal{L}_{k}v+b(x)v&=g(x,v)&in&\mathbb{R}^{n} \\
\end{array}\right.
$$
satisfies $||u||_{L^{\infty}(\mathbb{R}^{n})}\leq M$ and $M$ does not depend on $||u||$ (see Proposition \ref{prop45}). We emphasize that, in the best of our knowlegends, this result is being presented  for the first time in the literature. In \cite{alv2}, the authors have shown this result when the operator  $\mathcal{L}_{K}$ is the fractional Laplacian operator, that is, when $K(x)=\frac{C_{n,s}}{2}|x|^{-n-2s}$. But, in our case, we can not use the same technique used in \cite{alv2}, because we do not have a version of the $s$-harmonic extension for more general integro-differential operators. Therefore, we present an another technique.

\begin{remark}\label{rem41}
Let $\beta>1$. Define the real functions
$$
m(x):= (\lambda-1)(x^{\beta}+x^{-\beta})-\lambda(x^{\beta-1}+x^{1-\beta})+2
$$
and
$$
p(x):= (\lambda-1)(x^{\beta}+x^{-\beta})+\lambda(x^{\beta-1}+x^{1-\beta})-2.
$$
where $\lambda:=\frac{\beta^{2}}{2\beta-1}$. Then $m(x)\geq0$ and $p(x)\geq0$ for all $x> 0$.
\end{remark}We will omit the proof of the claim that appears in the Remark \ref{rem41}. 

Let $\beta>1$ and define
 $$
f(x)=x|x|^{2(\beta-1)}
$$
and
 $$
g(x)=x|x|^{\beta-1}.
$$
The functions $f$ and $g$ are continuous and differentiable for all $x \in \mathbb{R}$.   Consider $x,y\in \mathbb{R}$ with $x \neq y$. By Mean Value Theorem, there are $\theta_{1}(x,y)$, $\theta_{2}(x,y) \in \mathbb{R}$ such that
\begin{equation}\label{eq:12}
f^{'}(\theta_{1}(x,y))=\frac{f(x)-f(y)}{x-y}
\end{equation}
and
\begin{equation}\label{eq:13}
g^{'}(\theta_{2}(x,y))=\frac{g(x)-g(y)}{x-y},
\end{equation}
that is
$$
(2\beta-1)|\theta_{1}(x,y)|^{2(\beta-1)}=\frac{x|x|^{2(\beta-1)}-y|y|^{2(\beta-1)}}{x-y}
$$
and
$$
 \beta|\theta_{2}(x,y)|^{(\beta-1)}=\frac{x|x|^{(\beta-1)}-y|y|^{(\beta-1)}}{x-y}.
$$
Implying that
\begin{equation}\label{eq:9}
|\theta_{1}(x,y)|=\left(\frac{1}{2\beta-1}\frac{x|x|^{2(\beta-1)}-y|y|^{2(\beta-1)}}{x-y}\right)^{\frac{1}{2(\beta-1)}}
\end{equation}
and
\begin{equation}\label{eq:10}
|\theta_{2}(x,y)|=\left(\frac{1}{\beta}\frac{x|x|^{(\beta-1)}-y|y|^{(\beta-1)}}{x-y}\right)^{\frac{1}{(\beta-1)}}.
\end{equation}
We consider $\theta_{1}(x,x)=\theta_{2}(x,x)=0$ for all $x \in \mathbb{R}$.
\begin{remark}\label{rem42}
Note that $|\theta_{1}(x,y)|=|\theta_{1}(y,x)|$ and $|\theta_{2}(x,y)| = |\theta_{2}(y,x)|$ for all $x,y \in \mathbb{R}$.
\end{remark}

\begin{lemma}\label{lem43}
With the same notation, if $x \neq 0$ then
$$
|\theta_{1}(x,0)|\geq|\theta_{2}(x,0)|.
$$
\end{lemma}

\begin{proof}
By equalities \ref{eq:9} and \ref{eq:10}, we have
$$
|\theta_{1}(x,0)|=\left(\frac{1}{2\beta-1}\frac{x|x|^{2(\beta-1)}}{x}\right)^{\frac{1}{2(\beta-1)}} = \left(\frac{1}{2\beta-1}\right)^{\frac{1}{2(\beta-1)}}|x|
$$
and
$$
|\theta_{2}(x,0)|=\left(\frac{1}{\beta}\frac{x|x|^{\beta-1}}{x}\right)^{\frac{1}{(\beta-1)}} = \left(\frac{1}{\beta}\right)^{\frac{1}{(\beta-1)}}|x|.
$$
Thereby,
$$
|\theta_{1}(x,0)|\geq|\theta_{2}(x,0)|.
$$
\end{proof}

\begin{lemma}\label{lem44}
If $x,y \in \mathbb{R}$, then
$$
|\theta_{1}(x,y)|\geq|\theta_{2}(x,y)|.
$$
\end{lemma}

\begin{proof}
If $x = 0$ or $y = 0$ then the inequality was proved by Lemma \ref{lem43} and Remark \ref{rem42}. The case $x=y$ is trivial. We can suppose that $x \neq y$, $x \neq 0$ and $y \neq 0$. By equalities \ref{eq:9} and \ref{eq:10} we have that
$$
|\theta_{1}(x,y)|\geq|\theta_{2}(x,y)|
$$
if, and only if
$$
\left(\frac{1}{2\beta-1}\frac{x|x|^{2(\beta-1)}-y|y|^{2(\beta-1)}}{x-y}\right)^{\frac{1}{2(\beta-1)}} \geq \left(\frac{1}{\beta}\frac{x|x|^{\beta-1}-y|y|^{\beta-1}}{x-y}\right)^{\frac{1}{(\beta-1)}}.
$$
This last inequality is true if, and only if
$$
\frac{1}{2\beta-1}\frac{x|x|^{2(\beta-1)}-y|y|^{2(\beta-1)}}{x-y} \geq \frac{1}{\beta^{2}}\left(\frac{x|x|^{\beta-1}-y|y|^{\beta-1}}{x-y}\right)^{2}.
$$
This last inequality occurs if, and only if
$$
\lambda (x-y)(x|x|^{2(\beta-1)}-y|y|^{2(\beta-1)}) \geq \left(x|x|^{\beta-1}-y|y|^{\beta-1}\right)^{2},
$$
that is
$$
\lambda(|x|^{2\beta}-xy|y|^{2(\beta-1)}-yx|x|^{2(\beta-1)}+|y|^{2\beta})\geq |x|^{2\beta}-2xy|x|^{\beta-1}|y|^{\beta-1}+|y|^{2\beta}.
$$
But, we are supposing that $x\neq 0$ and $y\neq 0$, then the last inequality is equivalent to
\begin{equation}\label{eq:11}
\begin{array}{ll}
&\lambda\left[ \left(\frac{|x|}{|y|}\right)^{\beta} - \left(xy\frac{|y|^{\beta-2}}{|x|^{\beta}}\right)-\left(\frac{xy|x|^{\beta-2}}{|y|^{\beta}}\right)+\left(\frac{|y|}{|x|}\right)^{\beta}\right] \\
&\geq \left(\frac{|x|}{|y|}\right)^{\beta} - 2\frac{x}{|x|}\frac{y}{|y|} + \left(\frac{|y|}{|x|}\right)^{\beta}
\end{array}
\end{equation}
We will prove that the inequality \ref{eq:11} is true. If $x\cdot y>0$, then
$$
\begin{array}{ll}
&\lambda\left[ \left(\frac{|x|}{|y|}\right)^{\beta} - \left(xy\frac{|y|^{\beta-2}}{|x|^{\beta}}\right)-\left(\frac{xy|x|^{\beta-2}}{|y|^{\beta}}\right)+\left(\frac{|y|}{|x|}\right)^{\beta}\right] - \left(\frac{|x|}{|y|}\right)^{\beta} + 2\frac{x}{|x|}\frac{y}{|y|} - \left(\frac{|y|}{|x|}\right)^{\beta}
\\& =\lambda\left[ \left(\frac{|x|}{|y|}\right)^{\beta} - \left(\frac{|y|}{|x|}\right)^{\beta-1}-\left(\frac{|x|}{|y|}\right)^{\beta-1}+\left(\frac{|y|}{|x|}\right)^{\beta}\right] - \left(\frac{|x|}{|y|}\right)^{\beta} + 2 - \left(\frac{|y|}{|x|}\right)^{\beta} \\
& = (\lambda-1)\left[\left(\frac{|x|}{|y|}\right)^{\beta}+\left(\frac{|x|}{|y|}\right)^{-\beta}\right]-\lambda\left[\left(\frac{|x|}{|y|}\right)^{\beta-1}+\left(\frac{|x|}{|y|}\right)^{-\beta+1}\right]+2 \\
& = m(\frac{|x|}{|y|}).
\end{array}
$$
If $x \cdot y<0$, then
 $$
\begin{array}{ll}
& \lambda\left[ \left(\frac{|x|}{|y|}\right)^{\beta} - \left(xy\frac{|y|^{\beta-2}}{|x|^{\beta}}\right)-\left(\frac{xy|x|^{\beta-2}}{|y|^{\beta}}\right)+\left(\frac{|y|}{|x|}\right)^{\beta}\right] - \left(\frac{|x|}{|y|}\right)^{\beta} + 2\frac{x}{|x|}\frac{y}{|y|} - \left(\frac{|y|}{|x|}\right)^{\beta}\\
& =\lambda\left[ \left(\frac{|x|}{|y|}\right)^{\beta} + \left(\frac{|y|}{|x|}\right)^{\beta-1}+\left(\frac{|x|}{|y|}\right)^{\beta-1}+\left(\frac{|y|}{|x|}\right)^{\beta}\right]- \left(\frac{|x|}{|y|}\right)^{\beta} - 2 - \left(\frac{|y|}{|x|}\right)^{\beta}\\
& = (\lambda-1)\left[\left(\frac{|x|}{|y|}\right)^{\beta}+\left(\frac{|x|}{|y|}\right)^{-\beta}\right]+\lambda\left[\left(\frac{|x|}{|y|}\right)^{\beta-1}+\left(\frac{|x|}{|y|}\right)^{-\beta+1}\right]-2\\
& = p(\frac{|x|}{|y|}).
\end{array}
$$
By Remark \ref{rem41}, we have that $m(\frac{|x|}{|y|})\geq 0$ and $p(\frac{|x|}{|y|}) \geq 0$. This proves the inequality \ref{eq:11} and the Lemma \ref{lem44}.
\end{proof}

Our main result of this section will establishes an important estimate
involving the $L^{\infty}(\mathbb{R}^{n})$ norm of the solution $u$ of the auxiliary problem. It states that:
\begin{proposition}\label{prop45}
Leq $h \in L^{q}(\mathbb{R}^{n})$ with $q>\frac{n}{2s}$, and $v \in E \subset X$ be a weak solution of
$$
\left\{\begin{array}{rlll}
-\mathcal{L}_{k}v+b(x)v&=g(x,v)&in&\mathbb{R}^{n} \\
\end{array}\right.
$$
where $g$ is a continuous functions satisfying
$$
|g(x,s)|\leq h(x)|s| 
$$
for $s\geq 0$, $b$ is a positive function in $\mathbb{R}^{n}$ and $E$ is definded as in \ref{def21}. Then, there is a constant $M=M(q,||h||_{L^{q}})$ such that
$$
||v||_{\infty}\leq M||v||_{2^{\ast}_{s}}.
$$
\end{proposition}

\begin{proof}
Let $\beta>1$. For any $n \in \mathbb{N}$ define
$$
A_{n}=\left\{x \in \mathbb{R}^{n}; |v(x)|^{\beta-1}\leq n\right\}
$$
and
$$
B_{n}:=\mathbb{R}^{n}\setminus A_{n}.
$$
Consider
$$
\begin{array}{ll}
f_{n}(t):=&
\left\{\begin{array}{lll}
t|t|^{2(\beta-1)}&se&|t|^{\beta-1}\leq n\\
n^{2}t&se&|t|^{\beta-1}>n
\end{array}
\right.
\end{array}
$$
and
$$
\begin{array}{ll}
g_{n}(t):=&
\left\{\begin{array}{lll}
t|t|^{(\beta-1)}&se&|t|^{\beta-1}\leq n\\
nt&se&|t|^{\beta-1}>n
\end{array}
\right.
\end{array}
$$
Note that $f_{n}$ and $g_{n}$ are continuous functions and they are differentiable with the exception of $n^{\frac{1}{\beta-1}}$ and $-n^{\frac{1}{\beta-1}}$ and its derivatives are limited. Then $f_{n}$ and $g_{n}$ are lipschtz. Therefore, setting 
$$
v_{n}:=f_{n}\circ v
$$
and 
$$w_{n}:=g_{n}\circ v$$ 
we have that $v_{n},w_{n}\in E$.
Note that
$$
\begin{array}{ll}
\left[v,v_{n}\right]&=  \int_{A_{n}}\int_{A_{n}}(v_{n}(x)-v_{n}(y))(v(x)-v(y)K(x-y)dxdy \\ & +\int_{B_{n}}\int_{B_{n}}(v_{n}(x)-v_{n}(y))(v(x)-v(y)K(x-y)dxdy\\ &+2\left[v,v_{n}\right]_{A_{n}\times B_{n}}.
\end{array}
$$
By equality \ref{eq:12}, if $x,y \in A_{n}$ then
$$
v_{n}(x)-v_{n}(y)=f_{n}'(\theta_{1}(x,y))(v(x)-v(y)),
$$
where we are denoting $\theta_{1}(v(x),v(y))$ by $\theta_{1}(x, y)$.  Thereby,
\begin{equation}\label{eq:14}
\begin{array}{ll}
\left[v,v_{n}\right] &= \int_{A_{n}}\int_{A_{n}}(2\beta-1)|\theta_{1}(x,y)|^{2(\beta-1)}(v(x)-v(y))^{2}K(x-y)dxdy\\
 &+n^{2}\int_{B_{n}}\int_{B_{n}}(v(x)-v(y))^{2}K(x-y)dxdy +2\left[v,v_{n}\right]_{A_{n}\times B_{n}}.
\end{array}
\end{equation}
Analogously, by \ref{eq:13}
$$
\begin{array}{ll}
\left[w_{n},w_{n}\right] &=  \int_{A_{n}}\int_{A_{n}}\beta^{2}|\theta_{2}(x,y)|^{2(\beta-1)}(v(x)-v(y))^{2}K(x-y)dxdy\\
 &+n^{2}\int_{B_{n}}\int_{B_{n}}(v(x)-v(y))^{2}K(x-y)dxdy +2\left[w_{n},w_{n}\right]_{A_{n}\times B_{n}},
\end{array}
$$
where we are denoting $\theta_{2}(v(x),v(y))$ by $\theta_{2}(x, y)$. By Lemma \ref{lem44},
$$
\begin{array}{ll}
\left[w_{n},w_{n}\right]  &\leq \int_{A_{n}}\int_{A_{n}}\beta^{2}|\theta_{1}(x,y)|^{2(\beta-1)}(v(x)-v(y))^{2}K(x-y)dxdy\\
 &+n^{2}\int_{B_{n}}\int_{B_{n}}(v(x)-v(y))^{2}K(x-y)dxdy+2\left[w_{n},w_{n}\right]_{A_{n}\times B_{n}}.
\end{array}
$$
This implies that
\begin{equation}\label{eq:15}
\begin{array}{ll}
&\left[w_{n},w_{n}\right] + \int_{\mathbb{R}^{n}}b(x)w_{n}^{2}dx-\left[v,v_{n}\right]- \int_{\mathbb{R}^{n}}b(x)vv_{n}dx\\&\leq  (\beta-1)^{2}\int_{A_{n}}\int_{A_{n}}|\theta_{1}(x,y)|^{2(\beta-1)}(v(x)-v(y))^{2}K(x-y)dxdy\\
& +2\left[w_{n},w_{n}\right]_{A_{n}\times B_{n}}-2\left[v,v_{n}\right]_{A_{n}\times B_{n}}.
\end{array}
\end{equation}
The equation \ref{eq:14} implies that
\begin{equation}\label{eq:16}
\begin{array}{ll}
&\left[v,v_{n}\right]+\int_{\mathbb{R}^{n}}b(x)vv_{n}dx - 2\left[v,v_{n}\right]_{A_{n}\times B_{n}}\\ &\geq (2\beta-1)\int_{A_{n}}\int_{A_{n}}|\theta_{1}(x,y)|^{2(\beta-1)}(v(x)-v(y))^{2}K(x-y)dxdy,
\end{array}
\end{equation}
because $b(x)vv_{n}=b(x)w_{n}^{2}\geq 0$.
Replacing \ref{eq:16} in the inequality \ref{eq:15} we obtain
$$
\begin{array}{ll}
&\left[w_{n},w_{n}\right] + \int_{\mathbb{R}^{n}}b(x)w_{n}^{2}dx-\left[v,v_{n}\right]- \int_{\mathbb{R}^{n}}b(x)vv_{n}dx\\ &\leq \frac{(\beta-1)^{2}}{2\beta-1}\left([v,v_{n}]+\int_{\mathbb{R}^{n}}b(x)vv_{n}dx\right)\\
 &+2\left[w_{n},w_{n}\right]_{A_{n}\times B_{n}}+(-2-\frac{2(\beta-1)^{2}}{2\beta-1})\left[v,v_{n}\right]_{A_{n}\times B_{n}},
\end{array}
$$
that is
$$
\begin{array}{ll}
&\left[w_{n},w_{n}\right] + \int_{\mathbb{R}^{n}}b(x)w_{n}^{2}dx\\ &\leq (\frac{(\beta-1)^{2}}{2\beta-1}+1)\left([v,v_{n}]+\int_{\mathbb{R}^{n}}b(x)vv_{n}dx\right)\\
 &+2\left[w_{n},w_{n}\right]_{A_{n}\times B_{n}}+(-2-\frac{2(\beta-1)^{2}}{2\beta-1})\left[v,v_{n}\right]_{A_{n}\times B_{n}} \\
 &= \frac{\beta^{2}}{2\beta-1}\left([v,v_{n}]+\int_{\mathbb{R}^{n}}bvv_{n}dx\right)\\
 &+2\left[w_{n},w_{n}\right]_{A_{n}\times B_{n}}+(-2-\frac{2(\beta-1)^{2}}{2\beta-1})\left[v,v_{n}\right]_{A_{n}\times B_{n}} \\
 &\leq \beta \int_{\mathbb{R}^{n}}g(x,v)v_{n}dx\\
 &+2\left[w_{n},w_{n}\right]_{A_{n}\times B_{n}}+(-2-\frac{2(\beta-1)^{2}}{2\beta-1})\left[v,v_{n}\right]_{A_{n}\times B_{n}}.
\end{array}
$$
In short,
\begin{equation}\label{eq:17}
\begin{array}{ll}
\left[w_{n},w_{n}\right] &+ \int_{\mathbb{R}^{n}}b(x)w_{n}^{2}dx \leq \beta \int_{\mathbb{R}^{n}}g(x,v)v_{n}dx \\
 &+2\left[w_{n},w_{n}\right]_{A_{n}\times B_{n}}+(-2-\frac{2(\beta-1)^{2}}{2\beta-1})\left[v,v_{n}\right]_{A_{n}\times B_{n}}.
 \end{array}
\end{equation}
But, if $n \in \mathbb{N}$ and 
$$
C=2+\frac{2(\beta-1)^{2}}{2\beta-1},
$$
then a simple calculation shows that the function 
$$
r(s,t)=2(ns-t|t|^{\beta-1})^{2}-C(s-t)(n^{2}s-t|t|^{2(\beta-1)}),
$$
satisfies
\begin{equation}\label{eq24}
r(s,t)\leq 0
\end{equation}
for all $|s|>n^{\frac{1}{\beta-1}}$ and $|t|\leq n^{\frac{1}{\beta-1}}$. Then, taking $s=v(x)$ and $t=v(y)$ for $x\in B_{n}$ and $y \in A_{n}$ and replacing in \ref{eq24} we obtain

$$
2(w_{n}(x)-w_{n}(y))^{2}-C(v(x)-v(y))(v_{n}(x)-v_{n}(y))\leq 0.
$$
Thereby
$$
+2\left[w_{n},w_{n}\right]_{A_{n}\times B_{n}}+(-2-\frac{2(\beta-1)^{2}}{2\beta-1})\left[v,v_{n}\right]_{A_{n}\times B_{n}} \leq 0.
$$
By inequality \ref{eq:17},
\begin{equation}\label{eq:18}
\left[w_{n},w_{n}\right] + \int_{\mathbb{R}^{n}}b(x)w_{n}^{2}dx \leq  \beta \int_{\mathbb{R}^{n}}g(x,v)v_{n}dx.
\end{equation}
Let $S>0$ be the best constant verifying
$$
||u||_{2^{\ast}_{s}}^{2} \leq S||u||_{X}^{2},
$$
for all $u \in X$, that is
$$
S=\sup_{u \in X}\frac{||u||_{2^{\ast}_{s}}^{2}}{||u||_{X}^{2}}.
$$
By inequality \ref{eq:18},
$$
\begin{array}{ll}
\left(\int_{A_{n}}|w_{n}|^{2^{\ast}_{s}}dx\right)^{\frac{2}{2^{\ast}_{s}}}&\leq \left(\int_{\mathbb{R}^{n}}|w_{n}|^{2^{\ast}_{s}}dx\right)^{\frac{2}{2^{\ast}_{s}}}\\
&\leq S||w_{n}||^{2} \\
&  \leq S\beta\int_{\mathbb{R}^{n}}g(x,v(x))v_{n}dx \\
& \leq S\beta\int_{\mathbb{R}^{n}}h(x)w_{n}^{2}dx \\
 & \leq S\beta||h||_{q}||w_{n}||_{\frac{2q}{q-1}}^{2}.
\end{array}
$$
But, we have that $|w_{n}(x)|\leq |v(x)|^{\beta}$ for all $x \in B_{n}$ and $|w_{n}(x)|=|v(x)|^{\beta}$ for all $x \in A_{n}$. Thereby,
$$
\begin{array}{ll}
\left(\int_{A_{n}}|v|^{\beta 2^{\ast}_{s}}dx\right)^{\frac{2}{2^{\ast}_{s}}} \leq   S\beta||h||_{q}\left[\int_{\mathbb{R}^{n}}|v|^{\frac{2q\beta}{q-1}}dx\right]^{\frac{q-1}{q}}.
\end{array}
$$
By Monotone Convergence Theorem,
\begin{equation}\label{eq:19}
||v||_{2^{\ast}_{s}\beta} \leq (\beta S||h||_{q})^{\frac{1}{2\beta}}||v||_{2\beta q_{1}},
\end{equation}
where $q_{1}=\frac{q}{q-1}$.
Define
$$
\eta:=\frac{2^{\ast}_{s}}{2q_{1}}.
$$
and note that $\eta>1$. When $\beta=\eta$ we have that $2\beta q_{1} = 2^{\ast}_{s}$. Then, by \ref{eq:19}
\begin{equation}\label{eq:20}
||v||_{2^{\ast}_{s}\eta} \leq (\eta S||h||_{q})^{\frac{1}{2\eta}}||v||_{2^{\ast}_{s}}.
\end{equation}
Taking $\beta=\eta^{2}$ in \ref{eq:19} we obtain
\begin{equation}\label{eq210}
||v||_{2^{\ast}_{s}\eta^{2}} \leq \eta^{\frac{1}{\eta^{2}}} (S||h||_{q})^{\frac{1}{2\eta^{2}}}||v||_{2^{\ast}_{s}\eta}.
\end{equation}
By inequalities \ref{eq:20} and \ref{eq210} we have
$$
||v||_{2^{\ast}_{s}\eta^{2}} \leq \eta^{\frac{1}{(\eta^{2}}+\frac{1}{2\eta})} (S||h||_{q})^{(\frac{1}{2\eta^{2}}+\frac{1}{2\eta})}||v||_{2^{\ast}_{s}}.
$$
Inductively, we can prove that
$$
||v||_{2^{\ast}_{s}\eta^{m}} \leq \eta^{(\frac{1}{2\eta}+\frac{1}{\eta^{2}}+...+\frac{m}{2\eta^{m}})} (S||h||_{q})^{(\frac{1}{2\eta}+\frac{1}{2\eta^{2}}+...+\frac{1}{2\eta^{m}})}||v||_{2^{\ast}_{s}}
$$
for all $m \in \mathbb{N}$.
But,
$$
\sum_{i=1}^{\infty}\frac{m}{2\eta^{m}}=\frac{1}{2(\eta-1)^{2}}
$$
and
$$
\sum_{i=1}^{\infty}\frac{1}{2\eta^{m}}=\frac{1}{2(\eta-1)}.
$$
Thereby, for all $m>0$ 
$$
||v||_{2^{\ast}_{s}\eta^{m}} \leq \eta^{\frac{1}{2(\eta-1)^{2}}} (S||h||_{q})^{\frac{1}{2(\eta-1)}}||v||_{2^{\ast}_{s}}
$$
Recalling that
$$
||v||_{\infty}=\lim\limits_{n \rightarrow \infty}||v||_{p}
$$
and that $\eta>1$ we have that
$$
||v||_{\infty}\leq M||v||_{2^{\ast}_{s}} 
$$
for
$$
M=\eta^{\frac{1}{2(\eta-1)^{2}}} (S||h||_{q})^{\frac{1}{2(\eta-1)}}
$$
and
$$
\eta = \frac{n(q-1)}{q(n-2s)}
$$
We conclude the proof of Proposition \ref{prop45} noting that $M$ depends only on $q$, $||h||_{q}$.
\end{proof}

\section{Solution for Problem (P)}
 In this section, we prove the main result, the Theorem \ref{thm52}. By Corollary \ref{cor38}, there is $u \in X$ such that $J(u)=c$ and $J'(u)=0$. We have the following estimate for $||u||_{\infty}$.
\begin{lemma}\label{lem51}
The solution $u$ of the auxiliary problem satisfies
$$
||u||_{\infty} \leq M(2Skd)^{\frac{1}{2}}
$$
\end{lemma}

\begin{proof}
Consider the functions
$$
\begin{array}{lll}
H(x,t)&=&
\left\{
\begin{array}{lllll}
f(t)&if&|x|<R&or&f(t)\leq \frac{V(x)}{k}t\\
0&if&|x|\geq R&and&f(t)> \frac{V(x)}{k}t
\end{array}
\right.
\end{array}
$$
and
$$
\begin{array}{lll}
b(x)&=&
\left\{
\begin{array}{lllll}
V(x)&if&|x|<R&or&f(u)\leq \frac{V(x)}{k}u\\
\left(1-\frac{1}{k}\right)V(x)&if&|x|\geq R&and&f(u)> \frac{V(x)}{k}u.
\end{array}
\right.
\end{array}
$$
Note that the function $u$ is solution of
$$
\left\{\begin{array}{ll}
&-\mathcal{L}_{k}u+b(x)u=H(x,u)\\
&u \in E.
\end{array}\right.
$$
By $(f_{1})$,
$$
|H(x,t)|\leq c_{0}|t|^{p-1}
$$
for $p \in (2,2^{\ast}_{s})$. Thereby,
$$
|H(x,u)|\leq h(x)|u|,
$$
where $h(x)=C_{0}|u|^{p-2}$ . Note that $h \in L^{\frac{2^{\ast}_{s}}{p-2}}$ with
$$
||h||_{L^{q}(\mathbb{R}^{n})}\leq C(2ksd)^{\frac{p-2}{2^{\ast}_{s}}}.
$$
The number $p$ satisfies
$$
p<2^{\ast}_{s} = 2+\frac{2s}{n}2^{\ast}_{s}.
$$
Then
$$
q =\frac{2^{\ast}_{s}}{p-2}>\frac{n}{2s}.
$$
By Proposition \ref{prop45} and sobolev embedding
$$
||u||_{\infty}\leq M||u||_{2^{\ast}_{s}} \leq MS^{\frac{1}{2}}||u||,
$$
where $M=M(q,||h||_{q})$.
By Proposition \ref{prop39}, we have
\begin{equation}\label{eq535}
||u||_{\infty} \leq M(2kSd)^{\frac{1}{2}}.
\end{equation}
 \begin{theorem}\label{thm52}
Suppose that $V$ satisfies $(V_{1})$-$(V_{2})$ and that $f$ satisfies $(f_{1})$-$(f_{3})$. There is $\Lambda^{\ast}=\Lambda^{\ast}(V_{\infty},\theta,p,c_{0},s)>0$ such that if $\Lambda>\Lambda^{\ast}$ in $(V_{3})$, then the problem $(P)$ has a nonnegative nontrivial solution.
 \end{theorem}
Indeed, let $|x|\geq R$. If $u(x)=0$ then by denfition $f(u(x))=g(x,u(x))$. If $u(x)>0$ then
$$
\begin{array}{ll}
\frac{f(u(x))}{u(x)}&\leq c_{0}|u|^{p-2} \\
&\leq  c_{0}||u||_{\infty}^{p-2} \\
& = \frac{ c_{0}||u||_{\infty}^{p-2}}{\Lambda}\Lambda \\
& \leq \frac{k^{\frac{p}{2}}c_{0}M^{p-2}(2Sd)^{\frac{p-2}{2}}}{\Lambda}\frac{V(x)}{k}
\end{array}
$$
Define
$$
\Lambda^{\ast}=k^{\frac{p}{2}}c_{0}M^{p-2}(2Sd)^{\frac{p-2}{2}}.
$$
If $\Lambda>\Lambda^{\ast}$ then
$$
\frac{f(u(x))}{u(x)} \leq \frac{V(x)}{k}.
$$
By definition of $g$ we have $g(x,u(x))=f(u(x))$. Then $g(x,u(x))=f(u(x))$ for all $x \in \mathbb{R}^{n}$. Therefore, $u$ is a solution nonnegative and nontrivial of problem $(P)$.
\end{proof}

\end{document}